\documentclass[english, 10pt, a4paper]{article}
\PassOptionsToPackage{dvipsnames}{xcolor}
\usepackage{multirow}
\usepackage[hyphens]{url}
\usepackage{lipsum}
\usepackage{amsfonts}
\usepackage{graphicx}
\usepackage{epstopdf}
\usepackage[toc,page]{appendix}
\usepackage{algorithmic}
\ifpdf
  \DeclareGraphicsExtensions{.eps,.pdf,.png,.jpg}
\else
  \DeclareGraphicsExtensions{.eps}
\fi

\usepackage{enumitem}
\setlist[enumerate]{leftmargin=.5in}
\setlist[itemize]{leftmargin=.5in}

\PassOptionsToPackage{dvipsnames}{xcolor}
\usepackage{babel}
\usepackage[utf8]{inputenc}
\usepackage[T1]{fontenc}
\usepackage[top=2.5cm, bottom=2.5cm, left=2.5cm, right=2.5cm]{geometry} 
\usepackage{verbatim}
\usepackage{amsmath}
\usepackage{mathtools}
\mathtoolsset{showonlyrefs}
\usepackage{mwe,tikz}
\usepackage{amsthm}
\usepackage{amsfonts}
\usepackage{mathabx}
\usepackage{color}
\usepackage{xcolor}
\usepackage{amssymb}
\usepackage{mathrsfs}
\usepackage[ruled]{algorithm2e}
\usepackage{subfigure}
\usepackage{float}
\usepackage[most]{tcolorbox}
\usepackage{xr-hyper ,hyperref}
\usepackage{xspace}
\usepackage{regexpatch}
\usepackage{colortbl}%
\usepackage{makecell}
\usepackage{microtype}
\usepackage{pgfplots}
\usepackage{graphicx}
\usepackage{soul}
\usepackage{calc}
\usepackage{graphicx}

\DeclareMathAlphabet{\mathb}{OML}{cmm}{b}{it}

\newcommand{\R}{\mathbb{R}}		

\newcommand{\X}{\mathcal{X}}

\newcommand\innerprod[2]{\left\langle #1\,|\, #2 \right\rangle}
\newcommand{\fonc}[3]{#1:  #2  \rightarrow  #3}					
\newcommand{\syst}[1]{\left \{ \begin{array}{l} #1 \end{array} \right. \kern-\nulldelimiterspace}	
\newcommand{\prox}{\text{\normalfont prox}}

\newcommand{\epsprox}{\varepsilon\text{\normalfont-prox}}

\newcommand{\dom}{\text{\normalfont dom}\,}

\newcommand{\minimize}[2]{\ensuremath{\underset{\substack{{#1}}}{\mathrm{minimize}}\;\;#2 }}

\makeatletter
\newlength{\algorithmboxrule}
\newcommand{\algorithmboxcolor}{white}
\xpatchcmd*{\algocf@caption@boxruled}{0.0pt}{2\algorithmboxrule}{}{}
\xpatchcmd*{\algocf@caption@boxruled}{\vrule}{\vrule width \algorithmboxrule}{}{}
\xpatchcmd{\algocf@caption@boxruled}{\hrule}{\hrule height \algorithmboxrule}{}{}
\xpretocmd{\algocf@caption@boxruled}{\color{\algorithmboxcolor}}{}{}
\makeatother

\newtheorem{proposition}{Proposition}

\newtheorem{theorem}{Theorem}
\newtheorem{corollary}{Corollary}
\newtheorem{example}{Example}
\theoremstyle{definition}
\newtheorem{definition}{Definition}
\newtheorem{remark}{Remark}

\allowdisplaybreaks
\definecolor{shadecolor}{rgb}{0.8,0.8,0.8}

\title{Calculus rules for proximal $\varepsilon\text{-}$subdifferentials and inexact proximity operators for weakly convex functions}
\author{Ewa Bednarczuk\thanks{Warsaw University of Technology,  00-662 Warsaw, Koszykowa 75, Poland} \thanks{Systems Research Institute, PAS,  01-447 Warsaw, Newelska 6, Poland}
\and Giovanni Bruccola\footnotemark[2] \and Gabriele Scrivanti\thanks{Université Paris-Saclay, Inria, CentraleSupélec, CVN, 3 Rue Joliot Curie, 91190, Gif-Sur-Yvette, France.} 
\and The Hung Tran\footnotemark[2]}

\usepackage{amsopn}

\date{Spring 2022}

\usepackage{pgfplots}
\pgfplotsset{compat=1.18}

\begin{document}
\maketitle

\begin{abstract} 
 
We investigate proximal $\varepsilon\text{-}$subdifferentials
and derive sum rules that hold for weakly convex function,
by incorporating the corresponding  moduli of weak convexity into the respective formulas.
As an application, we 
analyse inexact proximity operators for weakly convex functions in terms of proximal $\varepsilon\text{-}$subdifferentials and the related notion of criticality.
\end{abstract}


\section{Introduction}
\label{sec:Introduction}
Proximal operators are a fundamental tool in constructing algorithms for solving large-scale convex optimisation problems \cite{Combettes2011}.   Responding to the need for solving  optimisation problems with convex objectives which do not fall into the class of typical convex objective functions appearing in data analysis (see \textit{e.g.}\ the webpage \cite{proxpage}), a number of inexact (approximate) proximal operators  have been  introduced, see  \textit{e.g.}\
\cite{salzo2012inexact, villa2013accelerated, rasch2020inexact}.

It is our aim  to investigate  inexact proximal operators for a class of functions which is larger than the one of convex functions: in the present work, we focus on weakly convex functions, which have been appearing in current models in data science problems in a rapidly growing number. 
Examples of weakly convex functions appearing in data analysis can be found in \cite{bohm2021variable, davis2019stochastic} and in the references therein. The growing interest in the use of this class of function in many fields of applications suggested the necessity of a careful analysis of their properties in terms of subdifferentials and proximal operators, which is the core of the present work.

\paragraph{Weak convexity}

Weak convexity can be considered as a special case of the general notions of $\gamma$-paraconvexity and $\alpha(\cdot)$-paraconvexity that were studied by, among others, Jourani and Rolewicz \cite{jourani1996Subdifferentiability,Rolewicz1979paraconvex,rolewicz2005paraconvex}. For a general characterisation in Hilbert spaces, weakly convex functions can be expressed as the difference between a convex function and a quadratic function. This class includes all the convex functions and all the smooth (but not necessarily convex) functions with a Lipschitz continuous gradient, together with many other interesting non-convex functions. 
\paragraph{Proximal Subdifferentials} 
It has already been observed  that the  concept of subdifferential which is particularly suitable when defining criticality for weakly convex functions  is that of \textit{proximal subdifferential}  (see 
{\em e.g.}\ \cite{Davis2017ProximallyGuidedStochasticSubgradient,Davis2018SubgradientMethodsSharpWeaklyConvex}). There exists a vast literature devoted to proximal subdifferentials, see {\em e.g.}\  in the finite dimensional case, the monograph by Rockafellar and Wets \cite{RockWets1998Variational}, in Hilbert spaces the work by Bernard and Thibault \cite{Bernard2005_ProxRegular}. 
 In these monographs and papers, the proximal subdifferential at a given $x_{0}$ is defined locally, in the sense that
  there exists a neighbourhood $V$ of $x_0$ and a constant $C\ge0$, such that for every $x\in V$,
\begin{equation}
\label{eq:intro:prox_sub}
    \langle x^*,x-x_0\rangle \leq f(x) - f(x_0) + C\|x-x_0\|^{2}.
\end{equation}
 
 In our developments we make use of a property that holds  in the class of paraconvex functions, called \textit{globalisation property}. Precisely, in the class of paraconvex functions, if inequality \eqref{eq:intro:prox_sub} holds in a neighbourhood of $x_0$, then it holds globally over the whole space for every $x\in\X$ with 
the same constant $C$ (see Prop. \ref{prop:globalisation} and Def. \ref{def:proximal_subdifferential}). This property allows us to treat the proximal subdifferential as a global tool - exactly as the convex subdifferential is treated with respect to convex functions - and deduce results in analogy to the convex setting.
More on the \textit{globalisation property} can be found in
 \cite{rolewicz1993globalization}.
 
We base our analysis on the more general notion of proximal $\varepsilon\text{-}$subdifferentials, which represents a useful tool allowing to take into account inexactness and perturbations in optimisation algorithms.
 
\paragraph{Contribution} 
Our contribution addresses the following issues.

\begin{enumerate}
    \item We start our analysis by providing sufficient and necessary conditions for the sum rule of the global proximal $\varepsilon\text{-}$subdifferentials for the sum of two $\rho\text{-}$weakly convex functions (see Theo. 
    \ref{theo: MReps} and Theo.
    \ref{theo: kruger:2}) (Section \ref{sec:Calculus_Rules}).
    We incorporate and make a consistent use of the modulus of proximal subdifferentiability and of the modulus of weak convexity $\rho$ into the calculus rules for proximal $\varepsilon\text{-}$subdifferentials. 

    \item 
    By using the above calculus rules, in Prop. \ref{cor:eps:prox:2} and Prop. \ref{cor:proxsub},
    we investigate the relationship between the ${\varepsilon\text{-proximal}}$ operator of a $\rho\text{-}$weakly convex function $f$ and the ${\varepsilon\text{-proximal}}$  subdifferential of $f$ (Section \ref{sec:inexact}).
    \item
     We relate the notion of inexact (approximate) proximal point that we infer to  \textit{Type}-1 and \textit{Type}-2 \textit{approximations} proposed, in the convex settings, by \cite{salzo2012inexact,rasch2020inexact}  (Section \ref{sec:inexact}).
\end{enumerate}

\section{Preliminaries}\label{sec:Preliminaries}
Before focusing on the class of weakly convex functions, we introduce a more general notion of ${\gamma\text{-paraconvexity}}$ and the corresponding notion of ${(\gamma,C)\text{-subdifferential}}$ as presented in \cite{jourani1996Subdifferentiability},  for $\gamma>0$ and $C\geq 0$. The class of $\gamma$-paraconvex functions has been studied in \cite{Rolewicz1979paraconvex}. For $\gamma=2$, we obtain weakly convex functions.\\
\begin{definition}[\textit{$\gamma$-Paraconvexity}]\label{def:paraconvexity}
Let $\mathcal{X}$ be a normed vector space. A function $\fonc{f}{\mathcal{X}}{(-\infty,+\infty]}$ is said to be $\gamma$\emph{-paraconvex} for $\gamma >0$, if there exists a positive constant $C$ such that for $\lambda\in[0,1]$, $(\forall (x,y)\in \mathcal{X}^2)$, the following inequality holds:
\begin{equation}
 f(\lambda x + (1-\lambda)y) \leq  \lambda f(x) + (1-\lambda) f(y) + C\lambda(1-\lambda)\|x-y\|^{\gamma}.
\end{equation}
\end{definition}
\vspace{0.2cm}

When $\alpha: [0,+\infty)\rightarrow [0,+\infty)$  is a non-decreasing function with
$\lim_{t\downarrow 0}\frac{\alpha(t)}{t}=~0$,
 a function ${\fonc{f}{\mathcal{X}}{(-\infty,+\infty]}}$ is called {\em $\alpha(\cdot)$-paraconvex} if there exists
a constant $C>0$ such that for $\lambda\in[0,1]$ 
\begin{equation}
\begin{aligned}
 &(\forall (x,y)\in \mathcal{X}^2)\qquad f(\lambda x + (1-\lambda)y)\\
 &\ \leq  \lambda f(x) + (1-\lambda) f(y) + C\min\{\lambda, 1-\lambda\}\alpha(\|x-y\|)
 \end{aligned}
\end{equation}
This class has been introduced by Rolewicz under the name of $\alpha(\cdot)$-strongly paraconvex functions and investigated in a series of papers by Jourani \cite{jourani1996open,jourani1996Subdifferentiability} and Rolewicz \cite{Rolewicz1979paraconvex,rolewicz2005paraconvex} . When $\alpha(\|x-y\|)=\|x-y\|^{\gamma}$ with $1<\gamma\leq 2$, the notion of $\alpha(\cdot)$-paraconvexity coincides with the one of $\gamma$-paraconvexity, (see \cite[Lemma 5]{rolewicz2001uniformly}). In Hilbert spaces, when 
$$
\limsup_{t\downarrow 0}\frac{\alpha(t)}{t^{2}}<+\infty,
$$
then a $\alpha(\cdot)$-paraconvex function is a difference of a convex and a quadratic function and is called {\em weakly convex} (see \cite{rolewicz2005paraconvex}). \\

\begin{definition}[$(\gamma,C)-$\textit{Subdifferential}{ \cite[Def. 3.1]{jourani1996Subdifferentiability}}]\label{def:Csubg_Jourani}
Let $\mathcal{X}$ be a normed vector space. By $\X^{*}$ we denote the  dual space of all continuous linear functionals defined on $\X$. Let $\gamma>0$ and $C>0$. 
Let ${\fonc{f}{\X}{(-\infty,+\infty]}}$ and $x_0 \in \dom f$. A point $x^*\in\mathcal{X}^*$ is said to be a $(\gamma,C)$-\textit{subgradient} of $f$ at $x_0$ if there exists a neighbourhood $V$ of $x_0$ such that the following \emph{subdifferential inequality} holds
\begin{equation}
\label{eq:def:prox_sub}
    (\forall x\in V)\quad \langle x^*,x-x_0\rangle \leq f(x) - f(x_0) + C\|x-x_0\|^{\gamma}.
\end{equation}
The set of all $(\gamma,C)-$subgradients
of $f$ at $x_0$ is denoted by $\partial_{(\gamma,C)}^{Loc}f(x_0)$ and it is referred to as $(\gamma,C)\text{-\textit{Subdifferential}}$. 
Whenever $\partial_{(\gamma,C)}^{Loc}f(x_0)\neq\emptyset$, we say that $f$ is \textit{proximally $C$-subdifferentiable} at $x_{0}$.\\
\end{definition}

{\begin{example}\label{ex:1paraconvexity}
    Let $\X = \R$. The function $f$ defined for every $x\in\X$ such as $f(x) = ||x|-1|$ is 1-paraconvex with $C=2$ (see \cite{jourani1996Subdifferentiability}). In addition, for $x_0 = 0$, we have that $\partial_{(1,2)}^{Loc} f(x_0) = [-1,1]$ since the following subdifferential inequality holds:
    \begin{flalign}
    \quad(\forall x \in \R)&& x^*x \leq ||x|-1|-1+ 2|x|. &&&
    \end{flalign}
    for every $x^* \in [-1,1]$.\\
    \end{example}
    
    It is interesting to notice that for the function in Example \ref{ex:1paraconvexity}, the subdifferential inequality from Def. \ref{def:Csubg_Jourani} at $x_0=0$ holds not only in a neighbourhood of the point, but on the whole space. For $\gamma$-paraconvex functions with $\gamma>1$, it is possible to show that the subdifferential inequality is always satisfied globally, meaning for every $x\in\X$, as stated by the following proposition:\\
    }
   
\begin{proposition}[ {\cite[Prop. 3.1]{jourani1996Subdifferentiability}}]\label{prop:globalisation}
Let $\mathcal{X}$ be a normed  space. Let ${\fonc{f}{\X}{(-\infty,+\infty]}}$ be $\gamma\text{-}$paraconvex with $\gamma>1$. Then there exists $C>0$ such that
\begin{equation}
    \partial_{(\gamma,C)}^{Loc}f(x_0) = \partial_{(\gamma,C)}f(x_0)
\end{equation}
where 
\begin{align*}
    \partial_{(\gamma,C)}f(x_0) := \left\{ x^*\in\mathcal{X}^*\,|\, \ \langle x^*,x-x_0\rangle \leq \right. \\
    \left. f(x) - f(x_0) + C\|x-x_0\|^{\gamma} \,\forall\ x\in \mathcal{X} \right\}.
\end{align*}
\end{proposition}
\vspace{0.2cm}

The constant $C$ can be the same as the one appearing in the definition of paraconvexity (Def. \ref{def:paraconvexity}). For similar results see {\cite[Prop. 3.5]{syga2019global}. }
Occasionally, we will refer to the constant $C$ appearing in the definition of proximal subdifferentiability as the {\em modulus of proximal subdifferentiability}.

  In the sequel we will use the global ${(2,C)\text{-subdifferential}}$ of $f$ at $x$ which 
  will be refferred to as \textit{proximal subdifferential}:
(see \textit{e.g.}\ \cite{RockWets1998Variational, Bernard2005_ProxRegular}).  \\

\begin{definition}[{Global Proximal Subdifferential}]\label{def:proximal_subdifferential}
Let $\mathcal{X}$ be a normed vector space. Let ${\fonc{f}{\X}{(-\infty,+\infty]}}$ and let {$x_0\in \dom f$}. Then the \textit{proximal subdifferential} of $f$ at $x_0$ with constant {$C\geq0$} is defined as the set 
\begin{equation}
\label{eq:def_prox_subdiff}
\begin{split}
    \partial_{(2,C)} f(x_0): = \{x^*\in\mathcal{X}^* \,|\, f(x)\geq f(x_0) \\+ \langle x^*,x-x_0\rangle  - C\|x-x_0\|^{2},\;\forall x\in \mathcal{X}\}.\\
    \end{split}
\end{equation}
\end{definition}
\vspace{0.2cm}

In view of \eqref{eq:def_prox_subdiff}, $\partial_{(2,0)}$ denotes the subdifferential in the sense of convex analysis. For simplicity, in this case we will use the notation $\partial_{\,0} \equiv \partial_{(2,0)}$.\\

When investigating inexact proximal points, the following concept of proximal $\varepsilon\text{-}$subdifferentials is used.\\

\begin{definition}[Global proximal $\varepsilon\text{-}$subdifferentials]\label{def:eps:prox_subdiff}
 Let $\mathcal{X}$ be a normed vector space and $\varepsilon\geq 0$. The global proximal $\varepsilon\text{-}$subdifferentials of a function ${\fonc{f}{\X}{(-\infty,+\infty]}}$ at $x_{0}\in\dom f$ for $C\geq 0$ is defined as follows:
 \begin{equation}
 \begin{split}
     \partial^{\,\varepsilon}_{(2,C)} f (x_0) = \{v\in\X\,|\, f(x)-f(x_{0})\ge \\\langle v, x-x_{0}\rangle-C\|x-x_{0}\|^{2}-\varepsilon\ \ \forall\ x\in\X\}.
     \end{split}
 \end{equation}
\end{definition}
\vspace{0.2cm}

 Clearly, for every $\varepsilon'\geq \varepsilon$ and $C'\geq C$ we have the following inclusion
 \begin{equation}
 \label{eq: mono}
     \partial^{\,\varepsilon}_{(2,C)} f (x_0) \subseteq \partial^{\,\varepsilon'}_{(2,C')} f (x_0).
\end{equation}

In Hilbert spaces, a weakly convex function $f$ in the sense of
Def. \ref{def:paraconvexity} for $\gamma = 2$ and $C = \rho/2$ can be characterised by the fact that $f(\cdot) + \rho/2\|\cdot\|^2$ is a convex function. A  proof can be  obtained by directly adapting the finite-dimensional proof given in {\cite[Prop. 1.1.3]{cannarsa2004semiconcave}}. Such a function will then be referred to as a $\rho$-\textit{weakly convex} and $\rho$ is known as \textit{modulus of weak convexity}.
A variant of Prop. \ref{prop:globalisation} corresponding to $\gamma=2$ and $C$ not necessarily coinciding with the weak convexity parameter can be found in \cite{syga2019global}.\\

For any  set-valued mapping $M: \X \rightrightarrows \X$, we will use the notation $\dom M$ to indicate the set
\begin{equation}
    \dom M := \{x\in \X\,|\, M(x) \neq \emptyset\},
\end{equation}
while for a function ${\fonc{f}{\X}{(-\infty,+\infty]}}$, the notation $\dom f$ will indicate the set
\begin{equation}
    \dom f := \{x\in \X\,|\, f(x) < +\infty\}.
\end{equation}

\begin{proposition}\label{prop: domain}Let $\X$ be a Hilbert space. Let ${\fonc{f}{\X}{(-\infty,+\infty]}}$ be a proper lower semicontinuous and  $\rho\text{-}$weakly convex function with $\rho \geq 0$. Then for every $\varepsilon\geq 0$
\begin{equation}
    \dom \partial_{(2,\rho/2)}^{\,\varepsilon} f = \dom \partial^{\,\varepsilon}_{0} (f + \frac{\rho}{2}\|\cdot\|^2)
\end{equation}
and 
\begin{equation}
   \dom \partial_{(2,\rho/2)}^{\,\varepsilon} f \subset \dom f.
\end{equation}
Moreover, for every $\varepsilon> 0$
\begin{equation}
    \dom \partial_{(2,\rho/2)}^{\,\varepsilon} f = \dom f.
\end{equation}
\end{proposition}
\vspace{0.2cm}
\begin{proof}
We start by showing  that for any ${x_0\in\dom \partial_{0}^{\,\varepsilon}(f + \frac{\rho}{2}\|\cdot\|^2) }$ and $\varepsilon\ge0$, we have that 
\begin{equation}
\label{eq: subdiff_eq:2}
    \partial_{0}^{\,\varepsilon}(f + \frac{\rho}{2}\|\cdot\|^2)(x_0) - \rho x_0 = \partial_{(2,\rho/2)}^{\,\varepsilon} f (x_0).
\end{equation}
Indeed, for any $v\in  \partial_0^{\,\varepsilon}\,(f+\frac{\rho}{2}\|\cdot\|^2)(x_0)$ and $x_0\in\dom \partial_{0}^{\,\varepsilon}(f + \frac{\rho}{2}\|\cdot\|^2)$, we have
\begin{equation}
\begin{aligned}
   &f(x)+\frac{\rho}{2}\|x\|^2-f(x_0)-\frac{\rho}{2}\|x_0\|^2\ge \langle v, x-x_0\rangle-\varepsilon\\
   &\quad\iff f(x)-f(x_0)\ge\\
   &\quad \langle v-\rho x_0, x-x_0 \rangle-\frac{\rho}{2}\|x-x_0\|^2-\varepsilon\\
   \end{aligned}
\end{equation}
which is equivalent to the fact that ${v-\rho x_0 \in\partial_{(2,{\rho}/{2})}^{\,\varepsilon} f(x_0)}$ and proves \eqref{eq: subdiff_eq:2}.
Hence, \linebreak ${\dom  \partial_{0}^{\,\varepsilon}(f + \frac{\rho}{2}\|\cdot\|^2)=\dom \partial_{(2,{\rho}/{2})}^{\,\varepsilon} f}$. 
Since $\partial_0^{\,\varepsilon}$ corresponds to the $\varepsilon\text{-}$subdifferential for convex functions and function $(f+ \frac{\rho}{2})(\cdot)$ is convex, we have that for all $\varepsilon>0$
\begin{equation}
\dom \partial_0^{\,\varepsilon}(f + \frac{\rho}{2}\|\cdot\|^2) = \dom (f + \frac{\rho}{2}\|\cdot\|^2)
\end{equation}
(see {\cite[Cor. 2.81]{barbu2012_convexity}}).  Since $\dom f=\dom (f+\frac{\rho}{2}\|\cdot\|^2)$, the assertion follows.
\end{proof}
\vspace{0.3cm}
In Prop. \ref{prop: domain}, we include the assumption of lower semicontinuity on $f$ because it is required by \cite[Cor. 2.81]{barbu2012_convexity}.\\
\begin{definition}[$\varepsilon\text{-}$solution]\label{def:eps_sol} Let $\X$ be a normed  space.
 Let ${\fonc{f}{\X}{(-\infty,+\infty]}}$ be a proper function that is bounded from below. Then, for any $\varepsilon\geq0$, the element $x_\varepsilon$ is said to be a $\varepsilon\text{-}$\textit{solution} to the minimisation problem
 \begin{equation}
     \minimize{x\in \X}{f(x)}
 \end{equation}
 if the following condition is satisfied:
 \begin{equation}
     \left(\forall x \in \X\right)\qquad f(x_\varepsilon) \leq f(x) + \varepsilon.\\
 \end{equation}
\end{definition}

\vspace{0.2cm}
\begin{definition}[$\varepsilon\text{-}C\text{-}$critical point]\label{def:var_crit_point} Let $\X$ be a normed  space and $\varepsilon\geq 0$. Let $\fonc{f}{\X}{ (-\infty,+\infty]}$ be a proper function. A point $x\in \X$ is said to be a $\varepsilon\text{-}C\text{-}$\textit{critical point} of $f$ if $0\in\partial^{\,\varepsilon}_{(2,C)} f(x)$. The set of $\varepsilon\text{-}C\text{-}$critical points is identified as 
\begin{equation}
    \varepsilon\operatorname{-crit}_{C}\; f := \{x\in\X\,|\,0\in \partial^{\,\varepsilon}_{(2,C)} f(x) \}.
\end{equation}

{When $\varepsilon=0$, we simply write $\operatorname{crit}_C f$.}
When $f$ is $\rho\text{-}$weakly convex, it is of particular interest to consider  $\varepsilon\text{-}\rho/2\text{-critical}$ points and then we write "$\varepsilon\text{-\textit{critical}}$ \textit{points}" and use the notation $\varepsilon\operatorname{-crit}$. \\
\end{definition}

{\begin{example}\label{ex: weak convexity}
    Let $\X = \R$. The function $f$ defined for every $x\in\X$ as $f(x) =  ||x|^2-1|$ is 2-paraconvex with $C=1$. In addition, for $x_0=0$ we have $0\in \partial_{(2,1)}f(x_0) $, since, for every $x\in\X$, function $f$ can be rewritten as 
    \begin{equation}
  f(x) = \max \{x^2-1,-x^2+1\}\\
\end{equation}
implying
\begin{flalign}
 \quad   (\forall x \in \R) &&  f(x) \geq  -x^2+1, &&&
\end{flalign}
and
\begin{flalign}
 \quad   (\forall x \in \R) &&  f(x) - f(x_0)\geq  -(x-x_0)^2 &&&
\end{flalign}
which shows that $x^*=0$ satisfies the global proximal subdifferential inequality from \eqref{eq:def_prox_subdiff}. Observe that $x_0 = 0$ is $x_0 \in \operatorname{crit}_1\, f$\\
    \end{example}}

\begin{remark}[Fermat's Rule]\label{rem:fermat}
We highlight that $\varepsilon\text{-}C\text{-}$criticality is a necessary condition for a point to be a $\varepsilon\text{-}$solution.
Notice that, under the assumptions of Prop. \ref{prop: domain}, $ \dom f = \dom  \partial^{\,\varepsilon}_{(2,C)} f$.
If $x_{\varepsilon} \in \dom f$ is a $\varepsilon\text{-}$solution of $f$, then
\begin{equation}
\begin{aligned}
    (\forall x\in \X) \qquad f(x) &\geq f(x_{\varepsilon}) - \varepsilon\\
    &\geq f(x_{\varepsilon}) - C \|x-x_{\varepsilon}\|^2- \varepsilon
    \end{aligned}
\end{equation}
for every $C\geq 0$. This implies $0\in \partial_{(2,C)}^{\,\varepsilon} f(x_{\varepsilon}).$
\end{remark}
\section{Calculus rules}\label{sec:Calculus_Rules}

In the literature, there exist numerous results providing calculus rules for the Fr\'echet, the limiting and the proximal subdifferentials, see \textit{e.g.}\ 
\cite{jourani1996Subdifferentiability,Kruger2003OnFS,Mordukhovich1995OnNS, mordukhovich2007exact,  Thibault1991} and many others. 
The main result of the present section is stated in Theo. \ref{theo: MReps}, where we provide the conditions for a sum rule for the global proximal $\varepsilon\text{-}$subdifferentials (in the sense of Def. \ref{def:eps:prox_subdiff}) of the sum of two weakly convex functions.
The proposed result allows us to extend the sum rule in \cite[Theo. 5.1]{jourani1996Subdifferentiability} -- 
proved for exact proximal subdifferentials in normed  spaces -- to proximal $\varepsilon\text{-}$subdifferentials in Hilbert spaces: the interesting aspect of such rule is that it allows us to trace the modulus of proximal subdifferentiability.

The following notion of $\rho$-conjugate function will be used in the proof of Theo. \ref{theo: MReps}.\\

\begin{definition}
Let $\X$ be a Hilbert space. Let ${\fonc{f}{\X}{(-\infty,+\infty]}}$ be a proper function. For every $\rho\geq 0$ the function $\fonc{(h)_{\rho}^*}{\X}{[-\infty,+\infty]}$ defined as 
\begin{equation}
    (f)^*_\rho(u) := \sup_{y\in \X} \left\{ -\frac{\rho}{2}\|y\|^2 + \langle u,y\rangle - f(y)\right\}
\end{equation}
is called $\rho$-\textit{conjugate} of $f$ at $u\in \X$ 
(when $\rho=0$ we obtain the definition of the conjugate as defined in convex analysis and in this case we omit the subscript).\\
\end{definition}
\vspace{0.2cm}

We recall the following result, which is an important fact in view of the proof of Theo. \ref{theo: MReps}.\\
\begin{theorem}[{\cite[Theo. 3]{rockafellar1966extension} }]
\label{th: rockafellar theorem 3}
Let $\X$ be a Hilbert space. Let $\fonc{f_0,\,f_1}{\X}{(-\infty,+\infty]}$ be proper convex functions.  Assume that $\dom f_0\cap \dom f_1$ contains a point at which either $f_0$ or $f_1$ is continuous.
Then, for all $s,\,x\in\X$ we have 
\begin{equation}
    \label{eq: rockafellar th 3}
    \begin{split}
        &(a)\quad \left(f_0+f_1 \right)^*(s)= \min_{\underset{s = p_0+p_1}{p_0,p_1\in \X}}\big\{f_0^*(p_0)+f_1^*(p_1)\big\}\\
        &(b)\quad \partial_0 \left(f_0+f_1 \right) (x)=\partial_0 f_0 (x)+\partial_0 f_1 (x)
    \end{split}
\end{equation}
\end{theorem}
\vspace{0.3cm}
\begin{remark}
By \cite[Theo. 15.3 (Attouch–Br\'ezis Theorem)]{Bauschke2017}, \cite{ attouch1986duality} , $(a)$ of Theo.  \ref{th: rockafellar theorem 3} can be proved under the assumption that 
  $f_0$ and $f_1$ are convex proper lsc functions such
that the conical hull of $\dom f_0- \dom f_1$ is a closed linear subspace, i.e.,

$$
0\in\text{sri}(\dom f_0 -\dom f_1),
$$
where sri denotes the strong relative interior, see \cite[Def. 6.9]{Bauschke2017}. The regularity assumption in \cite{attouch1986duality} is more general than the one in \cite{rockafellar1966extension} (see \cite[Remark 1.3]{attouch1986duality}). However, in some cases it is easier to verify the regularity condition requested in \cite{rockafellar1966extension}.
\end{remark}

The following proposition provides an auxiliary result used in the proof of Theo. \ref{theo: MReps}.\\
\begin{proposition}
\label{prop: conj} Let $\X$ be a Hilbert space.
 For $i=0,1$, let function ${\fonc{f_i}{\X}{(-\infty,+\infty]}}$ be proper lower semicontinuous and $\rho_i$-weakly convex on $\X$ with $\rho_i\geq 0$. Assume that $\dom f_0\cap \dom f_1$ contains a point at which either $f_0$ or $f_1$ is continuous. 
Then the following holds: for any $s\in \dom(f_0+f_1)^*_\rho$, there exist $p_0, p_1 \in \mathcal{X}$ such that $s=p_0+p_1$ and  \begin{equation}
    (f_0+f_1)^*_{\rho_0+\rho_1} (s) = (f_0)^*_{\rho_0} (p_0) + (f_1)^*_{\rho_1} (p_1).
\end{equation}
\end{proposition}
\vspace{0.2cm}

\begin{proof}
We have that $\overline{f}_{0}(\cdot) = f_0(\cdot)+\rho_0/2\|\cdot \|^2$ and $\overline{f}_{1}(\cdot) = f_1(\cdot)+\rho_1/2\|\cdot \|^2$ are convex. 
By Theo. \ref{th: rockafellar theorem 3},
there exist $p_0, p_1$ such that $s=p_0+p_1$ and  
\begin{equation}
    (\overline{f}_{0}+\overline{f}_{1})^* (s) = (\overline{f}_{0})^* (p_0) + (\overline{f}_{1})^* (p_1)
\end{equation}
Notice that for $i=0,1$
\begin{equation}
    \begin{aligned}
        \overline{f}_{i}^*(\cdot) &:= \sup_{y\in \X} \left\{ \langle \cdot,y\rangle - \overline{f}_{i}(y)\right\}\\
        &=\sup_{y\in \X} \left\{ \langle \cdot,y\rangle - f_i(y)-\frac{\rho_i}{2}\|y\|^2\right\}\\
       & = (f_i)_{\rho_i}^*(\cdot)
    \end{aligned}
\end{equation}
so
\begin{equation}
    (\overline{f}_{0}+\overline{f}_{1})^* (s) = (f_0)^*_{\rho_0} (p_0) + (f_1)^*_{\rho_1} (p_1)
\end{equation}
and in conclusion
\begin{equation}
    ({f}_{0}+{f}_{1})^*_{\rho_0+\rho_1} (s) = (f_0)^*_{\rho_0} (p_0) + (f_1)^*_{\rho_1} (p_1).
\end{equation}
\end{proof}
\vspace{0.2cm}

Now we are ready to prove the following  sum rule
 for proximal $\varepsilon\text{-}$subdifferentials $\partial_{(2,{\rho}/{2})}^{\,\varepsilon}$. This result generalises {\cite[Theo. 3.1.1]{hiriart2013convex}} and \cite[Theo. 2.8.7]{zalinescu2002convex}, which are formulated for convex functions and convex subdifferentials.  An important aspect of our result -- which will be used below in the analysis of proximal operators -- is that it allows us to trace the modulus of proximal subdifferentiability  (as related to the modulus of weak convexity of the functions involved). Indeed, a useful consequence of this is that whenever a weakly convex function $f$ is expressed as the sum of a convex function $f_0$ and a weakly-convex function $f_1$, the proximal subdifferential of $f$ corresponds to a (Minkowski) sum of the convex subdifferential of $f_0$ and of the proximal subdifferential of $f_1$. This will be exploited in the analysis of proximal operators in Section \ref{sec:inexact} \\

\begin{theorem}[Sum Rule for $\varepsilon\text{-}$subdifferential]
\label{theo: MReps} Let $\X$ be a Hilbert space. 
For $i=0,1$, let function $\fonc{f_i}{\X}{(-\infty,+\infty]}$ be proper lower semicontinuous and $\rho_i$-weakly convex on $\X$ with $\rho_i\geq 0$. Then, for all $x \in \dom f_0 \cap \dom f_1$ and for all $\varepsilon_0,\varepsilon_1 \geq 0$ we have
\begin{equation}
    \label{eq: inclusion MR: eps}
    \partial_{(2,\rho_0 / 2)}^{\,\varepsilon_0}f_0(x) +\partial_{(2,\rho_1 / 2)}^{\,\varepsilon _1}f_1(x) \subseteq \partial_{(2,{\rho}/{2})}^{\,\varepsilon}(f_0+f_1)(x)
\end{equation}
for all $\varepsilon\geq\varepsilon_0+\varepsilon_1$ and for all $\rho \geq \rho_0+\rho_1$.
The equality

\begin{equation}
\label{eq: equality MR: eps}
\begin{split}
    &\partial_{(2,(\rho_0+\rho_1)/2)}^{\,\varepsilon} (f_0 + f_1)(x)  =\\
    &\bigcup_{\varepsilon_{0},\varepsilon_{1}\,|\,\varepsilon_{0}+\varepsilon_{1}\leq \varepsilon} \partial_{(2,\rho_0/2)}^{\,\varepsilon_0}f_0(x) + \partial_{(2,\rho_1/2)}^{\,\varepsilon_1} f_1(x)
\end{split}
\end{equation}
holds when $\dom f_0\cap \dom f_1$ contains a point at which either $f_0+{\rho_0}/2\|\cdot \|^2$ or $f_1(\cdot)+{\rho_1}/2\|\cdot \|^2$ is continuous.\\
\end{theorem}
\begin{proof}
For $x\in\X$, if ${w\in\partial_{(2,\rho_0/2)}^{\,\varepsilon_0} f_0(x)}$, ${v\in\partial_{(2,\rho_1/2)}^{\,\varepsilon_1} f_1(x)}$, then it is clear that $w+v\in \partial_{(2,(\rho_0+\rho_1)/2)}^{\,\varepsilon} (f_0 + f_1)(x)$. Hence the inclusion \eqref{eq: inclusion MR: eps} is satisfied.

To prove the equality in \eqref{eq: equality MR: eps}, let us consider ${ x \in \dom f_0 \cap \dom f_1}$ and ${u\in\partial_{(2,\rho/2)}^{\,\varepsilon}(f_0+f_1)(x)}$, where ${\rho=\rho_0+\rho_1}$. By \cite[Theo. 2.4.ii, Eq. (5)]{Bednarczuk2021_OnDuality} we have 
\begin{equation}
\label{eq:BedSyg:new}
\begin{split}
 (f_0+f_1)(x) + (f_0+f_1)^*_\rho(u + \rho x)\\ \leq -\frac{\rho}{2}\|x\|^2 + \langle u +\rho x,x\rangle + \varepsilon
 \end{split}
\end{equation}

The inequality in \eqref{eq:BedSyg:new} implies that ${u + \rho x\in \dom (f_0+f_1)^*_\rho}$. By applying Prop. \ref{prop: conj}, there exist two elements ${p_0,p_1\in \X}$ such that ${u + \rho x = p_0 + p_1}$ and 
\begin{equation}
    (f_0+f_1)^*_\rho (u + \rho x) = f^*_{\rho_0}(p_0) + (f_1)^*_{\rho_1}(p_1)
\end{equation}
so that \eqref{eq:BedSyg:new} can be rewritten as 
\begin{equation}
\label{eq:BedSyg:new:2}
\begin{split}
 (f_0+f_1)(x) + (f_0)^*_{\rho_0}(p_0) + (f_1)^*_{\rho_1}(p_1) \\\leq -\frac{\rho}{2}\|x\|^2 + \langle p_0+p_1,x\rangle + \varepsilon
 \end{split}
\end{equation}
for all $ x \in \dom f_0 \cap \dom f_1$.
We now define the following values
\begin{align}\label{eq:eps0:new}
    \varepsilon_0:= f_0(x) + (f_0)_{\rho_0}^*(p_0) - \langle p_0, x \rangle + \frac{\rho_0}{2}\|x\|^2 \geq 0 \quad \\
    \label{eq:eps1:new}
    \varepsilon_1:= f_1(x) + (f_1)^*_{\rho_1}(p_1) - \langle p_1, x\rangle + \frac{\rho_1}{2}\|x\|^2 \geq 0 \quad
\end{align}
which are positive in view of the definition of $\rho$-conjugate.  
Notice that \eqref{eq:eps0:new} and \eqref{eq:eps1:new} can be rewritten as
\begin{equation}
\begin{split}
    \varepsilon_0 = f_0(x) + (f_0)^*_{\rho_0}((p_0- \rho_0 x) + \rho_0 x)\\ - \langle (p_0-\rho_0 x) + \rho_0 x, x\rangle + \frac{\rho_0}{2}\|x\|^2
    \end{split}
\end{equation}
\begin{equation}
\begin{split}
    \varepsilon_1 = f_1(x) + (f_1)^*_{\rho_1}((p_1- \rho_1 x) + \rho_1 x) \\- \langle (p_1-\rho_1 x) + \rho_1 x, x\rangle + \frac{\rho_1}{2}\|x\|^2
    \end{split}
\end{equation}
from which, by applying \cite[Theo. 2.4.ii, Eq. (5)]{Bednarczuk2021_OnDuality} in a similar fashion as in \eqref{eq:BedSyg:new}, we obtain 
\begin{equation}
   (p_0-\rho_0 x )\in \partial ^{\,\varepsilon_0}_{(2,\rho_0/2)}f_0(x)
\end{equation}
\begin{equation}
   (p_1-\rho_1 x )\in \partial ^{\,\varepsilon_1}_{(2,\rho_1/2)}f_1(x)
\end{equation}
which completes the proof.
\end{proof}
\vspace{0.3cm}
\begin{remark}
When $\varepsilon=0$, the sum rule presented in the theorem above can be shown to hold for $\gamma$-paraconvex fuctions, $\gamma>1$, defined over complete metric spaces (see {\cite[Theo. 5.1, Cor. 5.1]{jourani1996Subdifferentiability}}).  The sum rule in {\cite[Theo. 5.1, Cor. 5.1]{jourani1996Subdifferentiability}} also allows us to trace the modulus of subdifferentiability.
Our effort in Theo. \ref{theo: MReps} is to extend this result to the more general notion of $\varepsilon\text{-}$subdifferentials and, in order to do so, we exploited the structure of Hilbert spaces.
It is worth noticing that, as in the case of the sum rule for $\partial_{(\gamma, C)}$,
also in the case of the sum rule for $\partial^{\,\varepsilon}_{(\gamma, C)}$ we are able to control the modulus of proximal subdifferentiability. \\
\end{remark}
\begin{remark}
Let $f_i(x)$, $i=0,1$ be a $\rho_i$-weakly convex function.
If there exist $p_0, p_1$ such that 
${p_0+p_1=u\in \partial_{(2,(\rho_0+\rho_1)/2)}^{\,\varepsilon} (f_0 + f_1)(x)}$
and
$$
 p_0\in\partial_{(2,\rho_0/2)}^{\,\varepsilon_0}f_0(x) ,\quad p_1\in\partial_{(2,\rho_1/2)}^{\,\varepsilon_1} f_1(x)
$$
then by \eqref{eq: mono}, for all $\varepsilon_i'>\varepsilon_i$, $i=0,1$, we have
$$
p_0\in\partial_{(2,\rho_0/2)}^{\,\varepsilon'_0}f_0(x) ,\quad p_1\in\partial_{(2,\rho_1/2)}^{\,\varepsilon'_1} f_1(x).
$$

In particular, we can set $\varepsilon_1'+\varepsilon_2'=\varepsilon$ in \eqref{eq: equality MR: eps}.\\
\end{remark}

In the following theorem we show that, in presence of differentiable functions, the notion of proximal $\varepsilon\text{-}$subdifferentials allows us to infer an inclusion which involves the gradient of the differentiable function.\\

\begin{theorem}
\label{theo: kruger:2} Let $\X$ be a Hilbert space. 
Let ${\fonc{f_0}{\X}{(-\infty,+\infty]}}$ be proper, convex and differentiable with a $L_{0}$-Lipschitz continuous gradient on the whole space $\X$. Let ${\fonc{f_1}{\X}{(-\infty,+\infty]}}$ be proper and $\rho\text{-}$weakly convex on $\X$ with $\rho\ge 0$. 
Then, we  have the following inclusion
\begin{equation}
 \label{eq: prox inclusion 3}
    \partial^{\,\varepsilon}_{(2,\rho/2)}(f_{0}+f_{1})(x)\subset \nabla(f_{0})(x) +  \partial^{\,\varepsilon}_{(2,\rho/2+L_{0}/2)}f_{1}(x)
\end{equation}
for all $x \in \dom f_1$ such that $\partial^{\,\varepsilon}_{(2,\rho/2)}f_{1}(x)\neq\emptyset$.\\
\end{theorem}
\begin{proof}
In view of {\cite[Lemma 2.64]{Bauschke2017}}, both $f_{0}$ and $-f_{0}$ are proximally $L_{0}$-subdifferentiable on $\X$.
Precisely, for $f_{0}$ we have that for every $x,y\in \X$
\begin{align} 
\label{eq:sub_conv:2}
&(-f_{0})(y)-(-f_{0})(x)\ge\\ &\qquad \langle \nabla (-f_{0})(x)\ |\ y-x\rangle-\frac{L_0}{2}\|y-x\|^{2}.
\end{align}

Let us choose $x\in \dom f_1$. If ${\partial_{(2,\rho/2)}^{\,\varepsilon}(f_0+f_1)(x) = \emptyset}$, nothing needs to be proved. Otherwise, we take ${v\in \partial_{(2,\rho/2)}^{\,\varepsilon}(f_0+f_1)(x)}$. 
For every $y\in\X$,
 by adding  
\begin{equation}\label{eq:sub_conv:4}
(f_{1}+f_{0})(y)-(f_{1}+f_{0})(x)\ge \langle  v\ |\ y-x\rangle-\frac{\rho}{2}\|y-x\|^{2}-\varepsilon
\end{equation}
and \eqref{eq:sub_conv:2} we get
\begin{equation}
\partial^{\,\varepsilon}_{(2,\rho/2)}(f_{1}+f_{0})(x)+\nabla(-f_{0})(x)\subset \partial^{\,\varepsilon}_{(2,\rho/2+L_{0}/2)}f_{1}(x)
\end{equation}
\textit{i.e.}\
\begin{equation}
\partial^{\,\varepsilon}_{(2,\rho/2)}(f_{1}+f_{0})(x)\subset \partial^{\,\varepsilon}_{(2,\rho/2+L_{0}/2)}f_{1}(x)+\nabla(f_{0})(x).
\end{equation}

\end{proof}

\begin{remark}
We highlight the following facts.
\begin{itemize}
\item For $\varepsilon>0$, if $f_1$ is lower semicontinous, the non-emptiness of the proximal subdifferential is ensured by Prop. \autoref{prop: domain}, since $x\in \dom(f_0 + f_1)$.
\item For $\varepsilon>0$, {Theo. \ref{theo: kruger:2} still holds when $f_1$ is convex, i.e. ${\rho=0}$, but only if the proximal subdifferential is considered in the right hand side of the inclusion, as the proximal subdifferential satisfies the inclusion presented in \eqref{eq: mono}.}
In fact the inclusion becomes
$$
\partial^{\,\varepsilon}_{0}(f_{0}+f_{1})(x)\subset \nabla(f_{0})(x) +  \partial^{\,\varepsilon}_{(2,L_{0}/2)}f_{1}(x).
$$
\end{itemize}
\end{remark}
\vspace{ 0.3cm}

\section{inexact proximal maps}\label{sec:inexact}

In general settings, the computation of the proximal map needs to be addressed as an independent optimisation problem. 
Some practical examples involves
 non-convex $\ell_p$-seminorms (\textit{i.e.}\ $p\in(0,1)$) or the convex $\ell_p$-norms (\textit{i.e.}\ $p\geq1$), unless $p$ takes some specific values \cite{Chaux2007}.
Another example is given by the combination of a sparsity-promoting functions with a non-orthogonal linear operator, as in the case of the popular discrete Total Variation functional \cite{rudin1992nonlinear} (and its non-convex modifications), which has been extensively used in the context of image and signal processing. In these cases, at each point, the proximal map is defined up to a certain degree of accuracy and in the framework of proximal algorithms, it is important to carry out a convergence analysis that takes this fact into account. In order to do so, we consider the concept of $\varepsilon\text{-}$solution for an optimisation problem (see Def. \ref{def:eps_sol}) and the related notion of $\varepsilon\text{-}$proximal point.\\

\begin{definition}[$\varepsilon\text{-}$proximal point]\label{def:eps_prox} Let $\X$ be a normed vector space.
Let function $f:\X\rightarrow(-\infty,+\infty]$ be proper and bounded from below and $\alpha>0$. Then for all $y\in \X$ and for all $\varepsilon\geq 0$, any $\varepsilon\text{-}$solution to the proximal minimisation problem 
\begin{equation} 
\label{eq:epsmini}
\minimize{x\in\X}{ f(x)+\frac{1}{2 \alpha}\|x-y\|^{2},}
\end{equation}
 is said to be a $\varepsilon\text{-}$proximal point for $f$ at $y$ with respect to $\alpha$.  The set of all $\varepsilon\text{-}$proximal points of $f$ at $y$ with respect to $\alpha$ is denoted as
\begin{equation}
    \varepsilon\text{-}\prox_{\alpha f}(y) := \left\{x\in \X\,|\, x\  \text{is\;a\;}\varepsilon\text{-solution\;of\;\eqref{eq:epsmini}} \right\}\\
\end{equation}
\end{definition}
\vspace{0.2cm}

In the following result, we provide a relationship between the $\varepsilon\text{-}$proximal operator and the ${\varepsilon\text{-proximal}}$ subdifferentials of weakly convex function, using the sum rule from Theo. \ref{theo: kruger:2}. Specifically, in Prop. \ref{cor:eps:prox:2}, we trace the constant $C$ of the $(2,C)$-subdifferential of $f$.\\
\begin{proposition}
	\label{cor:eps:prox:2} 
	Let $\X$ be a Hilbert space.
		Let ${f:\X \to (-\infty,+\infty]}$ be a proper, lower semicontinuous $\rho\text{-}$weakly convex function that is bounded from below on $\X$. Let $\varepsilon\ge 0$, $\alpha>0$. Then for every $y\in\X$, $x_\varepsilon \in \epsprox_{\alpha f}(y)$ implies
		\begin{equation}\label{eq:eps:prox_inclusion:kruger}
		 \frac{y-x_{\varepsilon}}{\alpha} \in \partial^{\,\varepsilon}_{(2,\rho/2 + 1/(2\alpha))} f(x_{\varepsilon}).
		\end{equation}
 \end{proposition}
\vspace{0.2cm}

\begin{proof}
By the definition of $\varepsilon\text{-}$proximal point and Remark \ref{rem:fermat} we have
\begin{align}\label{eq:eps:implication}
    x_{\varepsilon} \in & \ \epsprox_{\alpha f}(y)\\ &\implies 0\in \partial^{\,\varepsilon}_{(2,\rho/2)}\left(\frac{1}{2\alpha}\|\cdot -y\|^2 + f(\cdot)\right) (x_\varepsilon).
\end{align}

The assumptions in Theo. \ref{theo: kruger:2} are satisfied since ${f_0(\cdot) = \frac{1}{2\alpha}\|\cdot - y\|^2}$ is differentiable on the whole space and its gradient has a Lipschitz constant $L_0= 1/\alpha$, hence we also have the inclusion 
\begin{align}
    0&\in \nabla \left(\frac{1}{2\alpha}\|\cdot - y\|^2\right)(x_{\varepsilon}) + \partial^{\,\varepsilon}_{(2,\rho/2 + 1/(2\alpha))} f(x_{\varepsilon}) \\
 &= \frac{x_{\varepsilon}-y}{\alpha} + \partial^{\,\varepsilon}_{(2,\rho/2+ 1/(2\alpha))} f(x_{\varepsilon}) 
     \end{align}
     which is equivalent to 
     \begin{align}
   \frac{y-x_{\varepsilon}}{\alpha} \in \partial^{\,\varepsilon}_{(2,\rho/2+ 1/(2\alpha))} f(x_{\varepsilon}).
\end{align}
\end{proof}

\begin{remark}
We highlight that \eqref{eq:eps:prox_inclusion:kruger} is related to the notion of Type-2 approximation of the proximal point that is proposed in
\cite{rasch2020inexact,salzo2012inexact}  in the convex settings. 
In other words, by using the ${\varepsilon\text{-proximal}}$ subdifferential instead of the 
(convex) ${\varepsilon\text{-subdifferential}}$, we can obtain a Type 2 approximation of the proximal point directly from Def. \ref{def:eps_prox}. 
This is due to Theo. \ref{theo: kruger:2}, at the expense of increasing the modulus of proximal subdifferentiability by $1/(2\alpha)$.
\end{remark}

\begin{remark}[$\varepsilon\text{-}$Subdifferential of a quadratic function]\label{rem:eps_sub_grad} Let $\X$ be a Hilbert space. As a consequence of \cite[Example 1.2.2]{hiriart2013convex}, for a function of the form 
$f_0(x) = \frac{1}{2\alpha}\|x-y\|^2$ for some $y\in \X$ and $\alpha>0$, we have 
\begin{equation}  \partial_0^{\,\varepsilon} f_0(x)  = \{  \frac{x-y}{\alpha} + \frac{e}{\alpha} \,|\, \frac{1}{2\alpha}\|e\|^2 \leq \varepsilon\}.
\end{equation}
\end{remark}
\vspace{0.2cm}

In view of Remark \ref{rem:eps_sub_grad} and Theo. \ref{theo: MReps}, we can provide another interpretation for the $\varepsilon\text{-}$proximal points of a $\rho\text{-}$weakly convex function in terms of proximal $\varepsilon\text{-}$subdifferentials, where this time we are able to trace the modulus of weak convexity of the function. \\
\begin{proposition}\label{cor:proxsub}	Let $\X$ be a Hilbert space.
		Let ${f:\X \to (-\infty,+\infty]}$ be a proper, lower semicontinuous $\rho\text{-}$weakly convex function that is bounded from below on $\X$ and let $\varepsilon\ge 0$, $\alpha>0$.
If $x_{\varepsilon} \in \varepsilon\operatorname{-prox}_{\alpha f}(y)$, then there exist $\varepsilon_0,\varepsilon_1 \geq 0$ with $\varepsilon_0 + \varepsilon_1 \leq \varepsilon$ and there exists $e \in \X $ with $\frac{\|e\|^2}{2\alpha}\leq \varepsilon_0$ such that
\begin{equation}\label{eq:eps:prox_inclusion}
\frac{y - x_\varepsilon - e}{\alpha } \in \partial_{(2, \rho/2)}^{\,\varepsilon_1} f(x_\varepsilon).
\end{equation}
		If {$\varepsilon=0$  and} $1/\alpha > \rho$, we obtain the equivalence.
\end{proposition}
\vspace{0.2cm}

\begin{proof}
By definition of $\varepsilon\text{-}$proximal point we have
\begin{align}
& x_{\varepsilon} \in \epsprox_{\alpha f}(y) \\
&\implies 0\in \partial^{\,\varepsilon}_{(2,\rho/2)}\left(\frac{1}{2\alpha}\|\cdot -y\|^2 + f(\cdot)\right) (x_\varepsilon).
\end{align}
We can now apply Theo. \ref{theo: MReps}, according to which there exist $\varepsilon_0,\varepsilon_1 \geq 0$ with  $\varepsilon_0 + \varepsilon_1 \leq \varepsilon$ such that
\begin{align}
\label{eq:inclusion_0}
   0\in \partial^{\,\varepsilon_0}_{0}\left(\frac{1}{2\alpha}\|\cdot -y\|^2\right)(x_\varepsilon) + \partial^{\,\varepsilon_1}_{(2,\rho/2)}f(x_\varepsilon).
\end{align}
By applying  Remark \ref{rem:eps_sub_grad}, we infer that there exists $e \in \X $ with $\frac{\|e\|^2}{2\alpha}\leq \varepsilon_0$ such that
\begin{equation}
    \frac{x_\varepsilon - y}{\alpha} + \frac{e}{\alpha} \in  \partial^{\,\varepsilon_0}_{0}\left(\frac{1}{2\alpha}\|\cdot -y\|^2\right)(x_\varepsilon)
\end{equation}
which implies
\begin{equation}
\frac{y - x_\varepsilon - e}{\alpha } \in \partial_{(2, \rho/2)}^{\,\varepsilon_1} f(x_\varepsilon).
\end{equation}

 {For the second part of the statement, we need to show that if $\frac{y-x_0}{\alpha} \in \partial_{(2,\rho/2)} f(x_{0})$,
then for every $x\in\X$ the following inequality holds
 \begin{flalign}
     && f(x_0) + \frac{1}{2\alpha}\|x_0-y\|^2 \leq f(x)  + \frac{1}{2\alpha}\|x-y\|^2. &&&
 \end{flalign}
 By definition of proximal subdifferential we have that for every $x\in\X$
 \begin{equation}
    f(x) \geq f(x_0) + \langle \frac{y-x_0}{\alpha},x-x_0\rangle -  \frac{\rho}{2} \|x-x_0\|^2 
 \end{equation}
 Since $1/\alpha > \rho$ we have that for every $x\in\X$
 \begin{equation}
    f(x) \geq f(x_0) + \langle \frac{y-x_0}{\alpha},x-x_0\rangle - \frac{1}{2\alpha}\|x-x_0\|^2
 \end{equation}
 We use the identity
  \begin{equation}
\begin{aligned}
    \innerprod{y-x_0}{x-x_0} &= -\frac{1}{2}\|y-x\|^2 + \frac{1}{2}\|x-x_0\|^2\\ &\qquad + \frac{1}{2}\|y-x_0\|^2,
\end{aligned}\end{equation}
which follows from \cite[Lemma 2.12]{Bauschke2017}. In conclusion, for every $x\in\X$ we have
\begin{equation}
 f(x_0)  + \frac{1}{2\alpha}\|y-x_0\|^2 \leq    f(x) +\frac{1}{2\alpha}\|y-x\|^2  .
 \end{equation}}
\end{proof}

\begin{remark}
The  interpretation provided by proposition \ref{cor:proxsub} is related to the notion of Type-1 approximation of the proximal point that is proposed in
\cite{rasch2020inexact,salzo2012inexact}  for convex functions.
\end{remark}

\begin{remark}
Notice that \eqref{eq:eps:prox_inclusion} from Prop. \ref{cor:proxsub} implies the inclusion \eqref{eq:eps:prox_inclusion:kruger} from Prop. \ref{cor:eps:prox:2}, which is based on Theo. \ref{theo: kruger:2}.
By definition of proximal $\varepsilon\text{-}$subdifferentials, \eqref{eq:eps:prox_inclusion} is equivalent to
\begin{align}
 &\left(x\in \X \right)\qquad  f\left(x\right)-f\left(x_{\varepsilon}\right)	\geq\\&\left\langle \frac{y-x_{\varepsilon}}{\alpha},x-x_{\varepsilon}\right\rangle -\frac{\rho}{2}\left\Vert x-x_{\varepsilon}\right\Vert ^{2} - \left\langle \frac{e}{\alpha}, x - x_{\varepsilon}\right\rangle -\varepsilon_1.
    \end{align}

We consider the following estimation
\begin{equation}
\begin{aligned}
    \left\langle \frac{e}{\alpha}, x - x_{\varepsilon}\right\rangle &\leq 
\varepsilon_0 + \frac{1}{2\alpha}\|x-x_{\varepsilon}\|^2
    \end{aligned}
\end{equation}
which stems from Cauchy-Schwarz and Young's inequality. It follows that
\begin{align}
     f\left(x\right)-f\left(x_{\varepsilon}\right)	
    &\geq\left\langle \frac{y-x_{\varepsilon}}{\alpha},x-x_{\varepsilon}\right\rangle\\&\qquad -\left(\frac{\rho}{2}
    + \frac{1}{2\alpha}\right)\left\Vert x-x_{\varepsilon}\right\Vert ^{2} -\varepsilon
    \end {align}
    that is equivalent to \eqref{eq:eps:prox_inclusion:kruger}
    \begin{equation}
    \frac{y-x_{\varepsilon}}{\alpha} \in \partial^{\,\varepsilon}_{(2,\rho/2 + 1/(2\alpha))} f(x_{\varepsilon}).
    \end{equation}
    \end{remark}
\vspace{0.2cm}

The inclusion in \eqref{eq:inclusion_0} further leads to the following corollary which is a generalisation of {\cite[Lemma 2]{rasch2020inexact}} from convex to proximal $\varepsilon\text{-}$subdifferentials.\\ 

\begin{corollary}\label{cor:proxsub:2}	Let $\X$ be a Hilbert space.
		Let ${f:\X \to (-\infty,+\infty]}$ be a proper, lower semicontinuous $\rho\text{-}$weakly convex function that is bounded from below on $\X$ and let $\varepsilon\ge 0$.
If $x_{\varepsilon} \in \varepsilon\operatorname{-prox}_{\alpha f}(x)$, then there exist  $e \in \X $ with $\frac{\|e\|^2}{2\alpha}\leq \varepsilon$ such that
\begin{equation}\label{eq:eps:prox_inclusion:2}
\frac{y - x_\varepsilon - e}{\alpha } \in \partial_{(2, \rho/2)}^{\,\varepsilon} f(x_\varepsilon)
\end{equation}
\end{corollary}
\vspace{0.2cm}
\begin{proof}
The proof is equivalent to the one from Prop. \ref{cor:proxsub}, with the only difference that we exploit the fact that the inclusion in \eqref{eq:inclusion_0} always implies the following inclusion
\begin{align}
\label{eq:inclusion_0:2}
   0\in \partial^{\,\varepsilon}_{0}\left(\frac{1}{2\alpha}\|\cdot -y\|^2\right)(x_\varepsilon) + \partial^{\,\varepsilon}_{(2,\rho/2)}f(x_\varepsilon) (x_\varepsilon).
\end{align}
by \eqref{eq: mono} and the fact that $\varepsilon_0$ and $ \varepsilon_1$ from Theo. \ref{theo: MReps} are always smaller than $\varepsilon$.
\end{proof}

\section{Conclusions}We discussed inexact proximal
operators (in the sense of Def. \ref{def:eps_prox}) for weakly convex functions defined on Hilbert spaces and their relationships to proximal $\varepsilon\text{-}$subdifferentials. We highlighted the main differences and similarities with the fully convex settings.
An important feature of the obtained results is that in Prop. \ref{cor:eps:prox:2} and Prop. \ref{cor:proxsub} we are able to control the moduli of proximal $\varepsilon\text{-}$subdifferentiability of $f$ (in relation to the moduli of weak convexity).
Our analysis provides tools to study the convergence of inexact proximal algorithms in which the proximal step is taken on a weakly convex function.
\\

\small{\textbf{Acknowledgments} This work has been supported by the ITN-ETN project TraDE-OPT funded by the European Union’s Horizon 2020 research and innovation programme under the Marie Skłodowska-Curie grant agreement No 861137. This work represents only the author’s view and the European Commission is not responsible for any use that may be made of the information it contains.}

\bibliographystyle{acm}
\bibliography{references}
\end{document}